\documentclass[11pt]{article}

\usepackage[T1]{fontenc}
\usepackage[utf8]{inputenc}
\usepackage{lmodern}
\usepackage{microtype}
\usepackage{amsmath,amssymb,amsthm,mathtools}
\usepackage{bm}
\usepackage{enumitem}
\usepackage{geometry}
\usepackage{hyperref}
\usepackage{mathrsfs}

\hypersetup{colorlinks=true, linkcolor=blue, citecolor=blue, urlcolor=blue}

\numberwithin{equation}{section}

\newtheorem{theorem}{Theorem}[section]
\newtheorem{proposition}[theorem]{Proposition}
\newtheorem{lemma}[theorem]{Lemma}
\newtheorem{corollary}[theorem]{Corollary}
\theoremstyle{definition}
\newtheorem{definition}[theorem]{Definition}
\newtheorem{remark}[theorem]{Remark}

\newcommand{\R}{\mathbb{R}}
\newcommand{\Z}{\mathbb{Z}}
\newcommand{\N}{\mathbb{N}}
\newcommand{\supp}{\operatorname{supp}}
\renewcommand{\Vec}{\operatorname{Vec}}
\newcommand{\CPwL}{\mathrm{CPwL}}
\newcommand{\ReLU}{\mathrm{ReLU}}
\newcommand{\Ups}{\Upsilon}
\newcommand{\one}{\mathbf{1}}

\title{Exact ReLU realization of tensor-product refinement iterates}
\author{Tsogtgerel Gantumur\\[0.5ex]
\small McGill University, Montr\'{e}al, QC, Canada\\
\small National University of Mongolia, Ulaanbaatar, Mongolia\\
\small Mongolian Academy of Sciences, Institute of Mathematics and Digital Technology\\[0.5ex]
\small \texttt{gantumur.tsogtgerel@mcgill.ca}}

\date{\today}

\begin{document}

\maketitle

\begin{abstract}
We study scalar dyadic refinement operators on \(\R^2\) of the form
\[
(Vf)(x,y)=\sum_{(j,k)\in\Z^2} c_{j,k}\,f(2x-j,\,2y-k),
\]
where only finitely many mask coefficients \(c_{j,k}\) are nonzero. Under a fixed support-window
hypothesis, we prove that for every compactly supported continuous piecewise linear seed
\(g:\R^2\to\R\), the iterates \(V^n g\) admit exact ReLU realizations of fixed width and depth \(O(n)\).

The proof gives a first genuinely two-dimensional extension of the exact realization theory for
refinement cascades. Using the one-dimensional exact loop-controller framework, it transports the
tensor-product residual dynamics exactly on the product of two polygonal loops and reduces the
remaining seam ambiguity to a final readout and selector step. The matrix cascade is then handled
by a fixed-depth recursive block, and general compactly supported \(\CPwL\) seeds are reduced to a
finite decomposition together with exact clamped gluing on the support window. This identifies the
tensor-product dyadic case as the natural first multivariate instance of the loop-controller method
for refinement iterates.
\end{abstract}

\tableofcontents

\section{Introduction}

\subsection{Background and motivation}

Neural-network approximation theory \cite{nnapprox} seeks structural explanations for why deep networks can
efficiently represent highly oscillatory, highly recursive, or highly self-similar functions. A
particularly clean result in this direction is the scalar binary theorem on refinable functions
\cite{refinable}: if
\[
(Vg)(x)=\sum_{j=0}^N c_j\,g(2x-j),
\]
and if \(g\) is compactly supported and continuous piecewise linear, then the iterates \(V^n g\)
admit exact ReLU realizations of fixed width and depth growing linearly in \(n\). In that setting,
the proof is organized around three ingredients: a cascade identity, a controlled treatment of
the residual dynamics, and a recursive fixed-depth update block for the finite-dimensional
cascade state. The result is one of the clearest examples in neural-network realization theory in
which linear growth of depth is explained by an intrinsic recursive analytic mechanism rather
than by a generic approximation argument.

The purpose of the present paper is to develop the first genuinely multivariate step beyond
that theorem. We keep the output scalar and the refinement homogeneous, but replace the
one-dimensional parameter by a two-dimensional one. Thus we study refinement operators of
the form
\begin{equation}
\label{eq:V-intro}
(Vf)(x,y)=\sum_{(j,k)\in\Z^2} c_{j,k}\,f(2x-j,\,2y-k),
\end{equation}
where only finitely many coefficients \(c_{j,k}\) are nonzero. Our goal is to understand whether,
for a compactly supported scalar \(\CPwL\) seed \(g:\R^2\to\R\), the iterates \(V^n g\) can still be
realized exactly by ReLU networks of fixed width and depth \(O(n)\).

The tensor-product dyadic case is the natural place to make this first multivariate step. On the
conceptual side, it is already a genuinely two-dimensional parameter-domain problem, while
still retaining the simplest multivariate residual structure. On the technical side, the
one-dimensional scalar binary theory may be organized in two exact ways. In the original
surrogate-based formalism \cite{refinable}, the discontinuous residual digits are replaced by continuous piecewise
linear surrogate mechanisms, and one keeps track of bad sets created at intermediate stages,
showing that their contribution is harmless at the terminal stage. It is conceivable that such a
formalism could also be extended to the tensor-product setting, but the bookkeeping would
become appreciably more involved. By contrast, the exact loop-controller formulation \cite{loop} transports
the residual orbit by a forward-exact loop state and confines the remaining seam ambiguity to
the terminal readout and selector stage. This is the simpler framework to extend to the present
setting: because the tensor-product residual dynamics is still coordinatewise, the controller space
becomes simply the product of two polygonal loops. Thus the forward residual transport remains
exact, and the genuinely new work is concentrated in the two-dimensional selector and gluing
geometry.

At the same time, the passage from one parameter variable to two is not merely cosmetic. In
one dimension, the discontinuities of the dyadic residual map occur at isolated dyadic points. In
two dimensions, the corresponding singular set is a union of dyadic grid lines, and the natural
seam region is therefore a union of thin vertical and horizontal strips. Thus the seam
bookkeeping becomes geometric rather than purely one-dimensional. Moreover, the naive
tensor-product bump \(h(x)h(y)\) is not \(\CPwL\), so one cannot simply tensorize the
one-dimensional special-hat argument. One must instead work with genuine compactly
supported \(\CPwL\) interior atoms.

Accordingly, the theorem proved below remains a tensor-product result, but it already exhibits
the first genuinely multivariate instance of the loop-controller method.

\subsection{Main theorem}

We write $\Ups_{W,L}(\ReLU;d,N)$ for the set of outputs of fully connected ReLU networks with
width $W$, depth $L$, input dimension $d$, and output dimension $N$.

Our main result is the following exact realization theorem.

\begin{theorem}
\label{thm:main}
Let \(V\) be the scalar dyadic refinement operator on \(\R^2\) given by
\[
(Vf)(x,y)=\sum_{(j,k)\in\Z^2} c_{j,k}\,f(2x-j,\,2y-k),
\]
where \(c_{j,k}=0\) for all but finitely many \((j,k)\). Let \(g:\R^2\to\R\) be compactly supported
and \(\CPwL\). Assume that there exists a fixed rectangular support window
\([0,L_1]\times[0,L_2]\) containing \(\supp g\) and preserved by \(V\), in the sense that
\[
\supp f\subset[0,L_1]\times[0,L_2]
\quad\Longrightarrow\quad
\supp(Vf)\subset[0,L_1]\times[0,L_2].
\]
Then there exist constants \(C_0,C_1>0\), depending only on the mask support, the preserved
support window, and the quantitative \(\CPwL\) complexity of \(g\), such that
\[
V^n g\in \Ups_{C_0,C_1 n}(\ReLU;2,1),
\qquad n\ge 1.
\]
In particular, the width remains bounded independently of \(n\), while the depth grows linearly
with the refinement depth.
\end{theorem}

As in the one-dimensional theory, the statement is an \emph{exact realization} theorem for the
finite refinement iterates themselves. We do not formulate a separate approximation theorem
for a refinable limit function. Such an approximation statement can be obtained by combining
Theorem~\ref{thm:main} with the corresponding convergence theory for the cascade, 
as in \cite{refinable}.

\subsection{Proof idea and scope}

The proof combines the finite-dimensional cascade formalism with an exact controller for the
two-dimensional residual dynamics. First, vectorization on the preserved support window reduces
the refinement step to a finite-dimensional matrix cascade on \([0,1]^2\). Second, the residual pair is
transported exactly by a torus controller obtained as the product of two one-dimensional loop
controllers. 
Third, for a compactly supported \(\CPwL\) profile \(H\) supported away from the boundary, the scalar factor \(H(R^n(\cdot))\) is recovered by four terminal readout branches.
Fourth, the matrix cascade is
propagated by a fixed-depth recursive block using one-dimensional loop-state selectors and the
product gadget. Finally, arbitrary compactly supported \(\CPwL\) seeds are treated by a finite
decomposition into such building blocks, together with translation covariance and clamped gluing on
the support window.

The present paper is intentionally limited to the tensor-product dyadic scalar homogeneous case.
This is the cleanest setting in which a genuinely multivariate theorem can be proved, and it already
shows that the loop-controller architecture survives passage from one parameter variable to two.
At the same time, the argument suggests broader extensions, including diagonal tensor-product
dilations in higher dimension, vector-valued outputs, and certain non-diagonal expanding matrices
whose suitable iterate becomes diagonal. These directions appear plausible, but would require
separate treatment.

\subsection{Organization of the paper}

Section~\ref{sec:preliminaries} fixes notation, establishes the dyadic tensor-product cascade
identity, and introduces the special atoms used later in the proof. Section~\ref{sec:torus-controller}
develops the exact torus controller obtained as the product of two one-dimensional loop controllers
and proves the four-branch scalar readout formula for the special-atom scalar factor.
Section~\ref{sec:recursive-block} constructs the recursive realization of the matrix cascade on the
unit square, proves the special-atom theorem, and glues the vectorized realization back to the full
support window. Section~\ref{sec:general-seeds} treats general compactly supported $\CPwL$ seeds
by finite special-atom decomposition and translation covariance, and deduces
Theorem~\ref{thm:main}. Section~\ref{sec:conclusions} concludes with a brief discussion of the
method and its broader multivariate outlook.

\section{Preliminaries and notation}
\label{sec:preliminaries}

This section fixes notation and records the finite-dimensional cascade formalism underlying the
whole paper. The basic mechanism is the same as in the one-dimensional theory,
but here it is recast for scalar functions on $\R^2$ and for tensor-product dyadic refinement.
The main point is that, once the support window is fixed, the refinement step can be encoded by
a finite family of transition matrices acting on the vector of unit-square patches of the function.

\subsection{Network classes and $\CPwL$ functions}

For integers $W,L,d,N\ge 1$, let
\(\Ups_{W,L}(\ReLU;d,N)\)
denote the set of outputs of fully connected feed-forward ReLU networks with width $W$, depth
$L$, input dimension $d$, and output dimension $N$.

We write $\CPwL$ for continuous piecewise linear. Thus a scalar function $f:\R^2\to\R$ is $\CPwL$
if every compact subset of $\R^2$ admits a finite polygonal subdivision on whose cells $f$ is affine.
In particular, every compactly supported $\CPwL$ function is determined by finitely many affine
pieces.

We shall use repeatedly the following standard closure properties.

\begin{remark}
\label{rem:standard-closure}
The following facts are standard.

\begin{enumerate}[label=(\roman*)]
\item Every fixed compactly supported $\CPwL$ function on $\R^2$ has an exact finite ReLU
realization.

\item Finite sums of fixed-width depth-$O(n)$ networks can be realized by enlarging the width by a
constant factor while preserving the depth bound $O(n)$.

\item Composing a depth-$O(n)$ network with a fixed affine input translation or with a fixed
$\CPwL$ output map does not change the fact that the depth remains $O(n)$.

\item If $F\in \Ups_{W,L}(\ReLU;2,N)$ and $\Phi:\R^N\to\R^M$ is fixed and $\CPwL$, then
$\Phi\circ F$ belongs to $\Ups_{W',L+L'}(\ReLU;2,M)$ for some constants $W',L'$, depending only
on $W$ and on the fixed local complexity of $\Phi$.
\end{enumerate}
\end{remark}

We shall not track optimal constants in these closure statements. Throughout the paper, the
quantity of primary interest is the asymptotic dependence on the refinement depth $n$, with all
other geometric and combinatorial data regarded as fixed.

\subsection{Dyadic digits and tensor-product residuals}
\label{ss:digits-residuals}

We begin with the one-dimensional dyadic digit and residual maps. For $t\in[0,1)$, define
\[
Q_1(t):=\lfloor 2t\rfloor \in \{0,1\},
\qquad
r(t):=2t-Q_1(t)\in[0,1).
\]
At the endpoint $t=1$ we adopt the convention
\(Q_1(1)=1\) and \(r(1)=1\).
Thus \(r\) is the usual doubling map on the circle, written on the closed interval \([0,1]\) with
the seam identification \(0\sim 1\).

For $z=(x,y)\in[0,1]^2$, define the dyadic digit map and the tensor-product residual map
\[
Q(z):=(Q_1(x),Q_1(y)),
\qquad
R(z):=(r(x),r(y)).
\]
Thus the residual dynamics is coordinatewise.
We define the successive digits and residuals by
\[
R_0(z):=z,
\qquad
R_n:=R^n \quad (n\ge 1),
\qquad
Q_j(z):=Q(R_{j-1}(z)) \quad (j\ge 1).
\]
Equivalently, we have
\[
R_n(z)=\bigl(r^n(x),\,r^n(y)\bigr),
\qquad
Q_j(z)=\bigl(Q_1(r^{j-1}(x)),\,Q_1(r^{j-1}(y))\bigr).
\]

\begin{remark}
\label{rem:residual-affine-2d}
For each $n\ge 1$ and each dyadic square
\[
I_{\bm m,n}:=
\Big[\frac{m_1}{2^n},\frac{m_1+1}{2^n}\Big)
\times
\Big[\frac{m_2}{2^n},\frac{m_2+1}{2^n}\Big),
\qquad
\bm m=(m_1,m_2)\in\{0,\dots,2^n-1\}^2,
\]
the residual iterate is affine:
\[
R_n(x,y)=\bigl(2^n x-m_1,\;2^n y-m_2\bigr),
\qquad
(x,y)\in I_{\bm m,n}.
\]
This is immediate from the one-dimensional identity
\[
r^n(t)=2^n t - m
\qquad
\text{on }
\Big[\frac{m}{2^n},\frac{m+1}{2^n}\Big),
\]
applied in each coordinate separately.
\end{remark}

\subsection{The refinement operator and the support window}

Fix finitely many coefficients $c_{j,k}\in\R$, not all zero, and define the scalar dyadic refinement
operator
\begin{equation}
\label{eq:operator-V}
(Vf)(x,y):=\sum_{(j,k)\in\Z^2} c_{j,k}\,f(2x-j,\,2y-k).
\end{equation}
Since only finitely many coefficients are nonzero, the sum is finite for every $f$.

Let $g:\R^2\to\R$ be compactly supported and assume
\(\supp g\subset [0,L_1]\times[0,L_2]\)
for some integers $L_1,L_2\ge 1$. We assume throughout that this rectangular support window is
preserved by the operator $V$, namely, that
\begin{equation}
\label{eq:support-preservation}
\supp f\subset [0,L_1]\times[0,L_2]
\quad\Longrightarrow\quad
\supp(Vf)\subset [0,L_1]\times[0,L_2].
\end{equation}
This is the natural hypothesis under which the refinement dynamics can be encoded in a fixed
finite-dimensional state space.

\subsection{Vectorization on unit squares}

For $1\le a\le L_1$ and $1\le b\le L_2$, define the $(a,b)$-patch of a function $f$ by
\[
f_{a,b}(u,v):=f(u+a-1,\;v+b-1),
\qquad
(u,v)\in[0,1]^2.
\]
Thus $f_{a,b}$ is the restriction of $f$ to the unit square
\([a-1,a]\times[b-1,b]\),
reparameterized back to the reference square $[0,1]^2$.

The vectorization of $f$ is the finite vector-valued function
\[
G(u,v):=\Vec(f)(u,v):=\bigl(f_{a,b}(u,v)\bigr)_{1\le a\le L_1,\;1\le b\le L_2}.
\]
We fix once and for all an ordering of the index pairs $(a,b)$; the precise ordering is irrelevant,
provided it is used consistently throughout. Thus $G$ takes values in $\R^{L_1L_2}$ and records
all unit-square pieces of $f$ simultaneously on the single reference square $[0,1]^2$.

For the iterates of the seed $g$, we write
\[
G_n:=\Vec(V^n g),
\qquad n\ge 0.
\]
In particular, note that \(G_0=\Vec(g)\).

\begin{remark}
\label{rem:vectorization-philosophy}
The point of vectorization is that the refinement step no longer acts on a scalar function defined on
the whole support window, but on a finite vector of reference-square patches. Once written in this
form, the refinement step becomes multiplication by a matrix selected by the current dyadic digit
pair. This is the tensor-product analogue of the one-dimensional vectorization used in the scalar
dyadic theory.
\end{remark}

\subsection{Block transition matrices}

For each dyadic digit
\(q=(q_1,q_2)\in\{0,1\}^2\),
define the transition matrix
\(T_q \in \R^{L_1L_2\times L_1L_2}\)
as follows. Its entry from source patch $(\alpha,\beta)$ to target patch $(a,b)$ is
\begin{equation}
\label{eq:Tq-entry}
(T_q)_{(a,b),(\alpha,\beta)}
=
c_{\,q_1+2(a-1)-(\alpha-1),\;q_2+2(b-1)-(\beta-1)}.
\end{equation}
Whenever the corresponding mask index is absent, the coefficient is interpreted as zero.

The meaning of \eqref{eq:Tq-entry} is simple. On the branch where the current digit pair is
$q=(q_1,q_2)$, the arguments
$2x-j$ and  $2y-k$
can be rewritten in terms of the common residual variables
$r(x)$ and $r(y)$,
plus integer translations. The matrix $T_q$ records exactly which source patch of the function
contributes to which target patch after this rewriting.

\begin{remark}
\label{rem:Tq-interpretation}
The matrices $T_q$ are the two-dimensional tensor-product analogue of the one-dimensional
matrices $T_0,T_1$ in the scalar dyadic theory, cf. \cite{refinable,loop}. Here there are four possibilities, corresponding to
the four dyadic quadrants of the unit square.
\end{remark}

\subsection{One-step cascade identity}

We now prove the basic finite-dimensional identity behind the whole construction.

\begin{proposition}[One-step tensor-product cascade identity]
\label{prop:one-step-cascade-2d}
For every $z\in[0,1]^2$, we have
\[
G_1(z)=T_{Q(z)}\,G(R(z)).
\]
\end{proposition}

\begin{proof}
Fix $z=(x,y)\in[0,1]^2$, and write
\(Q(z)=q=(q_1,q_2)\).
Let $(a,b)$ be a target patch index. By definition of vectorization, we have
\[
(G_1)_{a,b}(x,y)
=
(Vg)(x+a-1,\;y+b-1).
\]
Using the definition of the refinement operator \eqref{eq:operator-V}, we obtain
\[
(G_1)_{a,b}(x,y)
=
\sum_{(j,k)\in\Z^2}
c_{j,k}\,
g(2x+2(a-1)-j,\;2y+2(b-1)-k).
\]
Since
\(2x=r(x)+q_1\) and \(2y=r(y)+q_2\),
this becomes
\[
(G_1)_{a,b}(x,y)
=
\sum_{(j,k)\in\Z^2}
c_{j,k}\,
g(r(x)+q_1+2(a-1)-j,\;r(y)+q_2+2(b-1)-k).
\]

Now with
\(\alpha:=q_1+2(a-1)-j+1\) and \(\beta:=q_2+2(b-1)-k+1\), we have
\[
r(x)+q_1+2(a-1)-j = r(x)+\alpha-1,
\qquad
r(y)+q_2+2(b-1)-k = r(y)+\beta-1.
\]
Whenever
\(1\le \alpha\le L_1\) and \(1\le \beta\le L_2\),
the corresponding term is exactly
\(g_{\alpha,\beta}(R(z))\).
If either index lies outside the support window, then the corresponding physical point lies outside
$[0,L_1]\times[0,L_2]$, and the term vanishes by the support assumption.

Therefore we infer
\[
(G_1)_{a,b}(z)
=
\sum_{\alpha,\beta}
c_{\,q_1+2(a-1)-(\alpha-1),\;q_2+2(b-1)-(\beta-1)}
\,G_{\alpha,\beta}(R(z)).
\]
By the definition \eqref{eq:Tq-entry} of the transition matrix $T_q$, this is precisely the
$(a,b)$-component of $T_qG(R(z))$. Since $(a,b)$ was arbitrary, the proof is complete.
\end{proof}

\begin{corollary}[Iterated cascade identity]
\label{cor:iterated-cascade-2d}
For every $n\ge 1$ and every $z\in[0,1]^2$, we have
\[
G_n(z)=T_{Q_1(z)}T_{Q_2(z)}\cdots T_{Q_n(z)}\,G(R_n(z)).
\]
\end{corollary}

\begin{proof}
We argue by induction on $n$. The case $n=1$ is exactly
Proposition~\ref{prop:one-step-cascade-2d}.

Assume the formula holds for some $n\ge 1$. Then we have
\[
G_{n+1}(z)
=
\Vec(V^{n+1}g)(z)
=
\Vec\bigl(V(V^n g)\bigr)(z).
\]
Applying Proposition~\ref{prop:one-step-cascade-2d} to the seed $V^n g$, we obtain
\[
G_{n+1}(z)=T_{Q(z)}\,G_n(R(z)).
\]
Now use the induction hypothesis at the point $R(z)$:
\[
G_n(R(z))
=
T_{Q_1(R(z))}\cdots T_{Q_n(R(z))}\,G(R_n(R(z))).
\]
Taking into account that
\(Q_j(R(z))=Q_{j+1}(z)\) and
\(R_n(R(z))=R_{n+1}(z)\)
we have
\[
G_n(R(z))
=
T_{Q_2(z)}\cdots T_{Q_{n+1}(z)}\,G(R_{n+1}(z)).
\]
Substituting this into the previous display gives
\[
G_{n+1}(z)
=
T_{Q_1(z)}T_{Q_2(z)}\cdots T_{Q_{n+1}(z)}\,G(R_{n+1}(z)),
\]
which is exactly the required formula with $n+1$ in place of $n$.
\end{proof}

\subsection{The product gadget}
\label{subsec:product-gadget}

We record here the standard product gadget from \cite[Lemma 9]{refinable} that will be used repeatedly in the recursive block.
Its role is not to implement general multiplication of two variable network outputs, but to gate a
vector-valued state by a scalar selector and to annihilate states that are already zero.

\begin{lemma}[Product gadget]
\label{lem:product-gadget}
Let $N\in\N$ and $a>0$. Define
\[
\Pi_a(\lambda,y)
:=
-\ReLU(\lambda a\one-y)
-\ReLU((1-\lambda)a\one-\ReLU(-y))
+a\one,
\]
for $\lambda\in[0,1]$ and $y\in\R^N$, where $\one\in\R^N$ is the vector of all ones. Then
\[
\Pi_a\in \Ups_{2N+1,2}(\ReLU;N+1,N),
\]
and for every $y\in[-a,a]^N$ and every $\lambda\in[0,1]$ one has
\[
\Pi_a(1,y)=y,
\qquad
\Pi_a(0,y)=0,
\qquad
\Pi_a(\lambda,0)=0.
\]
\end{lemma}

\subsection{Special atoms}
\label{subsec:special-atoms}

Fix once and for all $0<\varrho<\frac12$. The specific value of $\varrho$ is not important; its role
is simply to keep the support of the distinguished atoms away from the boundary of the reference
square.

\begin{definition}[Special atom]
\label{def:special-atom}
A special atom is a compactly supported non-negative $\CPwL$ function
\(H:\R^2\to[0,\infty)\)
such that
\[
\supp H\subset[\varrho,1-\varrho]^2,
\qquad\text{and}\qquad
H\ \text{has finite polygonal complexity.}
\]
\end{definition}

The support condition is the two-dimensional replacement for the one-dimensional special-hat
condition. It is exactly what allows the terminal seam ambiguity of the loop readout to be absorbed
by vanishing near the boundary of the square.

\section{Exact torus controller and scalar readout}
\label{sec:torus-controller}

This section develops the controller and scalar readout used in the rest of the paper. The
underlying ingredient is the one-dimensional exact loop controller from \cite{loop}, which in the
present binary scalar setting can be written in a concrete form and then used coordinatewise. In
particular, the tensor-product residual dynamics on \([0,1]^2\) is transported exactly on the
product of two polygonal loops. Thus no forward iteration of scalar residual surrogates is needed;
the only remaining seam ambiguity comes from the loop parametrization and is handled at the
terminal readout stage. Since the later argument depends directly on this controller, we include a
short proof of the binary one-dimensional construction in the concrete form needed here.

\subsection{A polygonal loop for the one-dimensional residual}
\label{subsec:polygonal-loop}

Recall from \S\ref{ss:digits-residuals} the one-dimensional dyadic residual map
\(r(t)=2t-\lfloor 2t\rfloor\) for \(t\in[0,1]\),
together with the endpoint convention $r(1)=1$. 
Thus \(r\) is the usual doubling map on the circle, written on the closed interval \([0,1]\) with
the seam identification \(0\sim 1\).

We use a concrete polygonal model of this circle dynamics. 
Let \(a_0=(0,0)\), \(a_1=(1,1)\), and \(a_2=(1,0)\).
Let $\Delta\subset\R^2$ be the closed triangle with vertices $a_0,a_1,a_2$, and denote its boundary by
\(\Gamma=\partial\Delta\).
We parametrize $\Gamma$ by the continuous piecewise affine map
\(E:[0,1]\to\Gamma\)
defined by
\begin{equation}
\label{eq:E-loop}
E(t)=
\begin{cases}
(3t,3t), & 0\le t\le \frac13,\\[4pt]
(1,2-3t), & \frac13\le t\le \frac23,\\[4pt]
(3-3t,0), & \frac23\le t\le 1.
\end{cases}
\end{equation}
Then $E(0)=E(1)=a_0$, and $E$ is injective on $[0,1)$. Thus $\Gamma$ is a polygonal realization of
the residual circle, with the single seam identification encoded by the equality $E(0)=E(1)$.

\begin{proposition}[Exact loop controller]
\label{prop:1d-loop-controller}
There exists a fixed $\CPwL$ map
\(F:\R^2\to\R^2\)
such that
\[
F(E(t))=E(r(t)),
\qquad t\in[0,1].
\]
Moreover, $F$ admits an exact finite ReLU realization, with width and depth depending only on its
piecewise-affine complexity.
\end{proposition}

\begin{proof}
Define first a boundary map
\(F_\Gamma:\Gamma\to\Gamma\)
by
\(F_\Gamma(E(t)):=E(r(t))\) for \(t\in[0,1]\).
This is well defined because the only identification in the parametrization $E$ is $E(0)=E(1)$, and
\[
E(r(0))=E(0)=a_0=E(1)=E(r(1)).
\]
More explicitly, $F_\Gamma$ is given on the three sides of $\Gamma$ by
\[
F_\Gamma(z)=
\begin{cases}
(2s,2s), & z=(s,s),\ 0\le s\le \frac12,\\[4pt]
(1,2-2s), & z=(s,s),\ \frac12\le s\le 1,\\[4pt]
(2y-1,0), & z=(1,y),\ \frac12\le y\le 1,\\[4pt]
(1-2y,1-2y), & z=(1,y),\ 0\le y\le \frac12,\\[4pt]
(1,2x-1), & z=(x,0),\ \frac12\le x\le 1,\\[4pt]
(2x,0), & z=(x,0),\ 0\le x\le \frac12.
\end{cases}
\]
Thus $F_\Gamma$ is continuous and piecewise affine on $\Gamma$.

Now choose any continuous piecewise affine extension of $F_\Gamma$ from $\Gamma$ to the whole
plane. For example, one may triangulate $\Delta$ compatibly with the six boundary breakpoints
\[
\textstyle(0,0),\ (\frac12,\frac12),\ (1,1),\ (1,\frac12),\ (1,0),\ (\frac12,0),
\]
extend affinely on each subtriangle of $\Delta$, and then extend further outside $\Delta$ by any fixed
piecewise affine rule on a finite polygonal subdivision of a surrounding polygonal region. This yields a global
$\CPwL$ map
\(F:\R^2\to\R^2\)
such that
\(F|_\Gamma = F_\Gamma\).
Hence
\[
F(E(t))=F_\Gamma(E(t))=E(r(t)),
\qquad t\in[0,1].
\]
Finally, since \(F:\R^2\to\R^2\) is a \(\CPwL\) map, each of its two scalar components admits an
exact depth-\(2\) ReLU realization by \cite{relu-fem}. Realizing these two components in parallel
gives the stated network-realizability property for \(F\).
\end{proof}

Iterating the controller immediately yields exact transport of the residual orbit.

\begin{corollary}[Exact one-dimensional loop-state transport]
\label{cor:1d-loop-iterate}
For every $n\ge 0$ and every $t\in[0,1]$,
\[
F^n(E(t))=E(r^n(t)).
\]
\end{corollary}

\begin{proof}
The case $n=0$ is tautological. If the identity holds at level $n$, then
\[
F^{n+1}(E(t))
=
F(F^n(E(t)))
=
F(E(r^n(t)))
=
E(r^{n+1}(t))
\]
by Proposition~\ref{prop:1d-loop-controller}. This proves the claim by induction.
\end{proof}

\subsection{The torus controller}
\label{subsec:tensor-product-controller}

We now pass from the one-dimensional loop controller to the tensor-product setting relevant for the
present paper. Define the controller embedding
\[
\mathcal E:[0,1]^2\to \Gamma\times\Gamma\subset\R^4,
\qquad
\mathcal E(x,y):=\bigl(E(x),E(y)\bigr),
\]
and the controller update map
\[
\mathcal F:\R^4\to\R^4,
\qquad
\mathcal F(z_1,z_2):=\bigl(F(z_1),F(z_2)\bigr).
\]
Thus $\mathcal E$ embeds the residual point $(x,y)$ into the product loop $\Gamma\times\Gamma$, which is a topological torus,
and $\mathcal F$ updates the two loop coordinates independently.

\begin{proposition}[Exact torus controller]
\label{prop:exact-torus-controller}
For every $n\ge 0$ and every $(x,y)\in[0,1]^2$, we have
\[
\mathcal F^n(\mathcal E(x,y))=\mathcal E(R_n(x,y)).
\]
Equivalently, if we define the forward controller states by
\(z_n(x,y):=\mathcal F^n(\mathcal E(x,y))\),
then it holds that
\[
z_n(x,y)=\bigl(E(r^n(x)),\,E(r^n(y))\bigr),
\qquad n\ge 0.
\]
\end{proposition}

\begin{proof}
Recall from \S\ref{ss:digits-residuals} that
\(R_n(x,y)=(r^n(x),\,r^n(y))\).
The case $n=0$ is immediate from the definition of $\mathcal E$. Suppose the identity holds at stage
$n$. Then we see
\[
z_{n+1}(x,y)
=
\mathcal F(z_n(x,y))
=
\mathcal F\bigl(E(r^n(x)),\,E(r^n(y))\bigr)
=
\bigl(F(E(r^n(x))),\,F(E(r^n(y)))\bigr).
\]
Applying Corollary~\ref{cor:1d-loop-iterate} in each coordinate gives
\[
z_{n+1}(x,y)
=
\bigl(E(r^{n+1}(x)),\,E(r^{n+1}(y))\bigr)
=
\mathcal E(R_{n+1}(x,y)).
\]
This proves the result by induction.
\end{proof}

\begin{remark}
\label{rem:torus-controller-philosophy}
The point of Proposition~\ref{prop:exact-torus-controller} is that the forward residual dynamics is
now exact. No scalar surrogate residuals are iterated forward. All seam ambiguity is postponed to
the terminal readout and selector stage.
\end{remark}

Because both $\mathcal E$ and $\mathcal F$ are fixed $\CPwL$ maps,
Proposition~\ref{prop:exact-torus-controller} has the following immediate network-theoretic
consequence.

\begin{corollary}[Network realization of the forward controller]
\label{cor:controller-network}
There exist constants $C_0,C_1>0$, depending only on the fixed local complexities of $E$ and $F$,
such that
\[
z_n\in \Ups_{C_0,C_1 n}(\ReLU;2,4),
\qquad n\ge 0.
\]
\end{corollary}

\begin{proof}
The embedding $\mathcal E$ is a fixed $\CPwL$ map, hence it admits an exact finite ReLU realization of
constant width and depth. Likewise, $\mathcal F$ is a fixed $\CPwL$ map on $\R^4$, so it also admits an
exact finite ReLU realization of constant width and depth. Since
\(z_n=\mathcal F^n\circ \mathcal E\),
the map $z_n$ is realized by one fixed input block for $\mathcal E$, followed by $n$ copies of the
fixed update block for $\mathcal F$. This gives fixed width and depth $O(n)$.
\end{proof}

\subsection{One-dimensional terminal readouts}
\label{subsec:1d-terminal-readouts}

To use the exact controller, we must recover the one-dimensional residual coordinate
from a point on the loop. Since \(E(0)=E(1)\), there is no single globally exact readout on
\(\Gamma\). The corresponding two-readout device from \cite{loop}, which we simply recall here in
the form needed below, is as follows.

\begin{lemma}[One-dimensional terminal readouts]
\label{lem:1d-terminal-readouts}
Fix numbers
\(0<\bar\varepsilon<\varrho<\frac12\).
Then there exist $\CPwL$ maps
\(\rho^- , \rho^+ : \R^2\to [0,1]\)
with the following properties.

\begin{enumerate}[label=(\roman*)]
\item The readouts are exact away from complementary seam neighborhoods:
\[
\rho^-(E(t))=t \quad \text{for } t\in[\bar\varepsilon,1],
\qquad
\rho^+(E(t))=t \quad \text{for } t\in[0,1-\bar\varepsilon].
\]

\item For every one-dimensional special-hat $h$ satisfying
\(\supp h\subset[\varrho,1-\varrho]\),
one has the exact identity
\[
h(t)=\min\bigl\{h(\rho^-(E(t))),\,h(\rho^+(E(t)))\bigr\},
\qquad t\in[0,1].
\]

\item The maps $\rho^\pm\circ E:[0,1]\to[0,1]$ admit exact finite ReLU realizations.
\end{enumerate}
\end{lemma}

Figure~\ref{fig:1d-terminal-readouts} gives a schematic view of the complementary readouts
$\rho^{-}$ and $\rho^{+}$ from Lemma~\ref{lem:1d-terminal-readouts}.

\begin{figure}[ht]
\centering
\includegraphics[width=.7\textwidth]{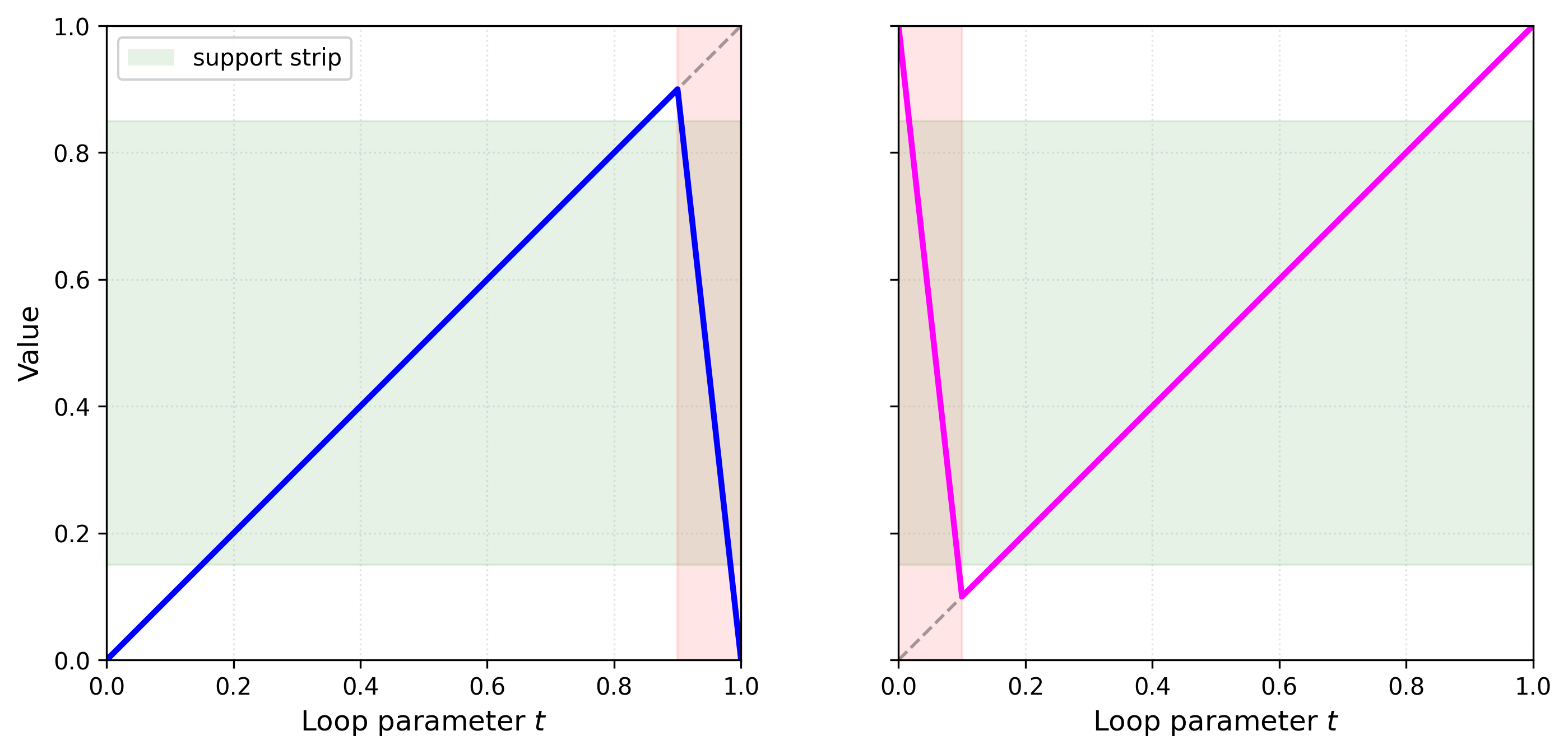}
\caption{Complementary scalar readouts \(\rho^{-}\) (left) and \(\rho^{+}\) (right). 
The dashed line is the identity and the solid curve is the corresponding readout. 
The horizontal shaded strip is the support interval \([\rho,1-\rho]\) of the special hat. 
The vertical shaded strip marks the seam interval modified by the readout: 
\([1-\varepsilon,1]\) on the left and \([0,\varepsilon]\) on the right.}
\label{fig:1d-terminal-readouts}
\end{figure}

\subsection{Four-branch scalar readout for special atoms}
\label{subsec:four-branch-scalar-readout}

We now combine the exact torus controller with the complementary one-dimensional readouts to
recover the scalar factor associated with a special atom.
For $\sigma,\tau\in\{-,+\}$, define the readout
\[
\rho^{\sigma,\tau}:\R^4\to\R^2,
\qquad
\rho^{\sigma,\tau}(z_1,z_2)
:=
\bigl(\rho^\sigma(z_1),\,\rho^\tau(z_2)\bigr).
\]

The next theorem is the two-dimensional scalar counterpart of the one-dimensional terminal
readout mechanism. It recovers the special-atom factor $H(R_n(\cdot))$ directly from the exact
torus-controller state by four terminal branches, corresponding to the two complementary seam
reads in each coordinate.

\begin{theorem}[Special-atom scalar factor]
\label{thm:special-atom-scalar-factor}
Let $H$ be a special atom in the sense of Definition~\ref{def:special-atom}. Then for every $n\ge 1$
and every $(x,y)\in[0,1]^2$, we have
\begin{equation}
\label{eq:H-Rn-min-four}
H(R_n(x,y))
=
\min_{\sigma,\tau\in\{-,+\}}
H\!\left(\rho^{\sigma,\tau}(z_n(x,y))\right),
\end{equation}
where $z_n(x,y)$ is the exact torus-controller state from
Proposition~\ref{prop:exact-torus-controller}.

Consequently, there exist constants $C_0,C_1>0$, depending only on $\bar\varepsilon$, $\varrho$,
and the fixed local complexity of $H$, such that
\[
H(R_n(\cdot))\in \Ups_{C_0,C_1 n}(\ReLU;2,1).
\]
\end{theorem}

\begin{proof}
Set
\(u=r^n(x)\) and \(v=r^n(y)\),
so that
\(R_n(x,y)=(u,v)\).
By Proposition~\ref{prop:exact-torus-controller}, we have
\(z_n(x,y)=\bigl(E(u),E(v)\bigr)\).
Hence, for every $\sigma,\tau\in\{-,+\}$, we infer
\[
\rho^{\sigma,\tau}(z_n(x,y))
=
\bigl(\rho^\sigma(E(u)),\,\rho^\tau(E(v))\bigr).
\]
Therefore \eqref{eq:H-Rn-min-four} is equivalent to
\begin{equation}
\label{eq:min-four-uv}
H(u,v)=\min_{\sigma,\tau\in\{-,+\}}
H\bigl(\rho^\sigma(E(u)),\,\rho^\tau(E(v))\bigr).
\end{equation}

We prove \eqref{eq:min-four-uv} by a case split.

\medskip
\noindent\emph{Case 1: $(u,v)\in[\bar\varepsilon,1-\bar\varepsilon]^2$.}
In this case both readouts are exact in each coordinate, so
\[
\rho^-(E(u))=u=\rho^+(E(u)),
\qquad
\rho^-(E(v))=v=\rho^+(E(v)).
\]
Thus all four branches coincide with $H(u,v)$, and therefore
\[
\min_{\sigma,\tau\in\{-,+\}}
H\bigl(\rho^\sigma(E(u)),\,\rho^\tau(E(v))\bigr)=H(u,v).
\]

\medskip
\noindent\emph{Case 2: $(u,v)\notin[\bar\varepsilon,1-\bar\varepsilon]^2$.}
Then at least one of the coordinates $u,v$ belongs to
\([0,\bar\varepsilon]\cup[1-\bar\varepsilon,1]\).
Because $\bar\varepsilon<\varrho$ and
\(\supp(H)\subset[\varrho,1-\varrho]^2\),
we have
\(H(u,v)=0\).

Now choose the signs so that the readout is exact in each coordinate:
\[
\sigma_*=
\begin{cases}
+, & u\in[0,\bar\varepsilon],\\
-, & u\in[1-\bar\varepsilon,1],\\
\text{arbitrary}, & u\in[\bar\varepsilon,1-\bar\varepsilon],
\end{cases}
\qquad
\tau_*=
\begin{cases}
+, & v\in[0,\bar\varepsilon],\\
-, & v\in[1-\bar\varepsilon,1],\\
\text{arbitrary}, & v\in[\bar\varepsilon,1-\bar\varepsilon].
\end{cases}
\]
Then we have
\[
\rho^{\sigma_*}(E(u))=u,
\qquad
\rho^{\tau_*}(E(v))=v,
\]
and hence
\[
H\bigl(\rho^{\sigma_*}(E(u)),\,\rho^{\tau_*}(E(v))\bigr)=H(u,v)=0.
\]
Since $H\ge 0$, all four branch values are nonnegative. Therefore
\[
\min_{\sigma,\tau\in\{-,+\}}
H\bigl(\rho^\sigma(E(u)),\,\rho^\tau(E(v))\bigr)=0=H(u,v).
\]
This proves \eqref{eq:min-four-uv}, and therefore \eqref{eq:H-Rn-min-four}.

For the network-realizability statement, note first that the controller state
\((x,y)\mapsto z_n(x,y)\)
belongs to $\Ups_{C_0',C_1' n}(\ReLU;2,4)$ by Corollary~\ref{cor:controller-network}. For each
$\sigma,\tau\in\{-,+\}$, the readout
\(\rho^{\sigma,\tau}:\R^4\to\R^2\)
is fixed and $\CPwL$, and $H:\R^2\to\R$ is a fixed compactly supported $\CPwL$ map. Hence each
branch
\((x,y)\mapsto H\!\left(\rho^{\sigma,\tau}(z_n(x,y))\right)\)
belongs to a class
\(\Ups_{\widetilde C_0,\widetilde C_1 n}(\ReLU;2,1)\)
with constants depending only on $\bar\varepsilon$, $\varrho$, and the fixed local complexity of $H$.

Finally, the minimum of two scalar outputs is realized by a fixed-depth ReLU gadget, 
and hence the minimum of the four branch outputs is obtained by composing this gadget twice.
Running the four branches in parallel changes the width only by a constant factor and adds only
$O(1)$ extra depth. Therefore there exist constants $C_0,C_1>0$, depending only on
$\bar\varepsilon$, $\varrho$, and the fixed local complexity of $H$, such that
\(H(R_n(\cdot))\in \Ups_{C_0,C_1 n}(\ReLU;2,1)\).
\end{proof}

The scalar part of the construction is now complete. The exact torus controller transports the
residual pair forward without approximation, and
Theorem~\ref{thm:special-atom-scalar-factor} shows that the terminal scalar factor associated with
a special atom can be recovered exactly from that controller state. It remains to treat the
finite-dimensional matrix cascade: the selector matrices, the recursive propagation of the vectorized
state, and the reconstruction of the physical function on the support window. This is the content of
the next section.

\section{Recursive realization for special atoms}
\label{sec:recursive-block}

In this section we turn from the scalar controller/readout mechanism to the finite-dimensional
matrix cascade. By Theorem~\ref{thm:special-atom-scalar-factor}, the scalar factor
\(H(R_n(\cdot))\)
associated with a special atom is already under exact control. What remains is to propagate the
vectorized cascade through the dyadic digit pairs and then to glue the resulting squarewise
realization back to a scalar function on the support window.

For the purposes of the present section, it is enough to treat a special atom
\(H:[0,1]^2\to\R\) supported in the reference square. The translated case will be handled later in
Section~\ref{sec:general-seeds} by translation covariance, so no shifted atoms are needed here.
Since \(H\) is supported in
the first unit square, its vectorization is simply
\[
G(u,v):=\Vec(H)(u,v)=H(u,v)\,b_{11},
\qquad (u,v)\in[0,1]^2 ,
\]
where \(b_{11}\in\R^{L_1L_2}\) denotes the standard basis vector corresponding to the patch \((1,1)\). 
Accordingly, Corollary~\ref{cor:iterated-cascade-2d} gives
\begin{equation}
\label{eq:G-n}
G_n(z)=T_{Q_1(z)}T_{Q_2(z)}\cdots T_{Q_n(z)}\,H(R_n(z))\,b_{11},
\qquad z\in[0,1]^2.
\end{equation}
Thus the scalar factor has already been separated off, and the remaining task is to realize the
matrix product by a fixed-depth recursive block.

\subsection{Binary loop selectors and terminal boundary localization}
\label{subsec:selectors-boundary-trap}

Unlike the terminal scalar readout from Section~\ref{sec:torus-controller}, the selector transition
width must depend on the final depth \(n\). Indeed, a point that enters the selector transition set at
an intermediate stage is not expected to remain near the boundary under further iteration; the
dyadic residual map is expanding. What matters is the terminal effect: \(\delta_n\) is chosen so that,
if one coordinate enters the transition set at some stage \(j\le n\), then by time \(n\) the terminal
residual lies in a boundary layer where the special atom already vanishes.

We recall the selector mechanism from \cite{loop}, specialized to the binary setting.
Fix once and for all a parameter
\(0<\bar\delta<1\),
and for each \(n\ge 1\) define
\(\delta_n:=\bar\delta\,\varrho\,2^{-n}\).
We then set
\(J_n:=[0,\delta_n]\cup[\frac12,\frac12+\delta_n]\).
This is the one-dimensional transition set used in the selector stage.

For each \(n\ge 1\), we fix \(\CPwL\) selectors
\(\chi_{0,n},\chi_{1,n}:\R^2\to[0,1]\)
with the following properties:

\begin{enumerate}[label=(\roman*)]
\item 
\(\chi_{0,n}(z)+\chi_{1,n}(z)=1\) for \(z\in\Gamma\);
\item 
\(\chi_{0,n}(E(t))=\chi_{[0,1/2)}(t)\) and \(\chi_{1,n}(E(t))=\chi_{[1/2,1]}(t)\) for every \(t\in[0,1]\setminus J_n\);
\item 
each \(\chi_{q,n}\) admits an exact finite ReLU realization with width and depth bounded
independently of \(n\), while the corresponding weights and biases may be chosen with magnitudes
bounded by \(C\,2^n\), where \(C\) depends only on \(\bar\delta\) and \(\varrho\).
\end{enumerate}

The next lemma is the key geometric fact that absorbs selector ambiguity. The one-dimensional
transition set \(J_n\) gives rise in the square to thin vertical and horizontal transition strips, since
ambiguity occurs whenever at least one coordinate is in transition. If a residual orbit enters these
strips at any stage, then the terminal residual is already forced outside the support box
\([\varrho,1-\varrho]^2\) of the special atom. Figure~\ref{fig:square} illustrates this geometry and a
sample residual orbit.

\begin{figure}[ht]
\centering
\includegraphics[width=.45\textwidth]{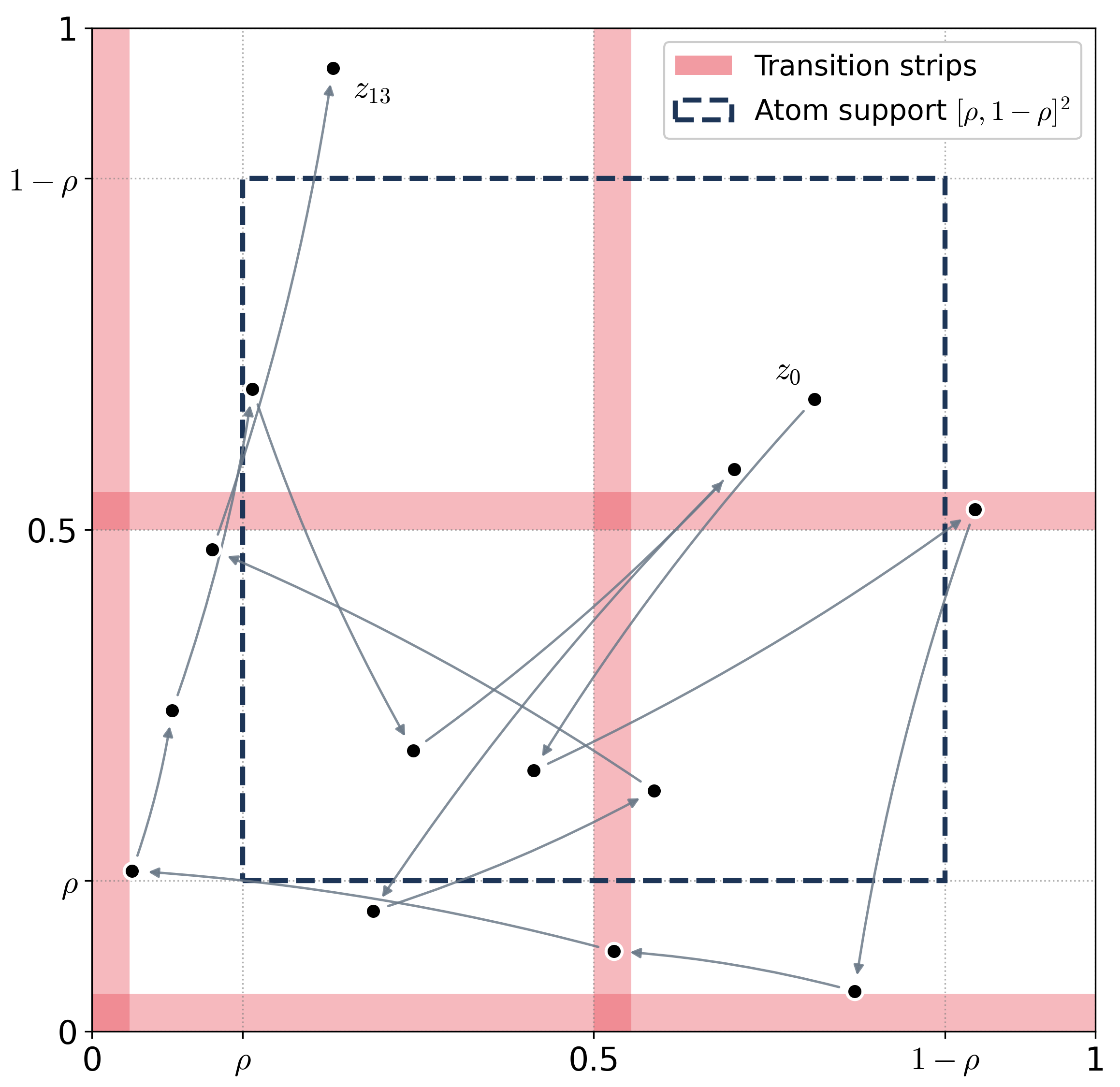}
\caption{Transition strips, support box, and a sample residual orbit.}
\label{fig:square}
\end{figure}

\begin{lemma}[Boundary localization]
\label{lem:bad-stage-kills-H}
Assume that for some \(j\in\{1,\dots,n\}\) one has
\[
r^{j-1}(x)\in J_n
\qquad\text{or}\qquad
r^{j-1}(y)\in J_n.
\]
Then we have
\[
H(R_n(x,y))=0.
\]
\end{lemma}

\begin{proof}
By symmetry, it suffices to treat only the \(x\)-coordinate.
So assume from now on that
\(r^{j-1}(x)\in J_n=[0,\delta_n]\cup[\frac12,\frac12+\delta_n]\).
If \(r^{j-1}(x)\in[0,\delta_n]\), then the dyadic digit is \(0\), and therefore
\[
r^j(x)=2r^{j-1}(x)\in[0,2\delta_n].
\]
If \(r^{j-1}(x)\in[\frac12,\frac12+\delta_n]\), then the dyadic digit is \(1\), and we now have
\[
r^j(x)=2r^{j-1}(x)-1\in[0,2\delta_n].
\]
Thus in all cases, we conclude
\(r^j(x)\in[0,2\delta_n]\).

We now iterate forward. Since the residual map is one-sided from the right at both breakpoints
\(0\) and \(\frac12\), each further iterate preserves the left boundary layer and doubles its width.
Hence
\[
r^n(x)\in[0,2^{\,n-j+1}\delta_n].
\]
Using
\(\delta_n=\bar\delta\,\varrho\,2^{-n}\),
we obtain
\[
2^{\,n-j+1}\delta_n
=
2^{\,n-j+1}\bar\delta\,\varrho\,2^{-n}
=
\bar\delta\,\varrho\,2^{1-j}
\le \bar\delta\,\varrho
< \varrho.
\]
Therefore we infer
\(r^n(x)\in[0,\varrho)\),
and in particular,
\[
R_n(x,y)=\bigl(r^n(x),r^n(y)\bigr)\notin[\varrho,1-\varrho]^2.
\]
Since
\(\supp H\subset[\varrho,1-\varrho]^2\),
it is immediate that
\(H(R_n(x,y))=0\).
\end{proof}

\subsection{Modified matrix selectors}
\label{subsec:modified-matrices}

For each stage \(j=1,\dots,n\), define the modified matrix field
\[
\widehat M_j(x,y)
:=
\sum_{q_1,q_2\in\{0,1\}}
\chi_{q_1,n}\!\bigl(z^{(1)}_{j-1}(x,y)\bigr)\,
\chi_{q_2,n}\!\bigl(z^{(2)}_{j-1}(x,y)\bigr)\,
T_{(q_1,q_2)},
\]
where
\[
z_{j-1}(x,y)=\bigl(z^{(1)}_{j-1}(x,y),z^{(2)}_{j-1}(x,y)\bigr)
=
\bigl(E(r^{j-1}(x)),E(r^{j-1}(y))\bigr)
\]
is the exact torus-controller state from Proposition~\ref{prop:exact-torus-controller}.

Thus \(\widehat M_j(x,y)\) is a \(\CPwL\) convex combination of the four digit-pair transition
matrices. If neither coordinate has entered the transition set at the preceding stages, we shall call
the orbit {\em good}; in that case exactly one digit pair is active and \(\widehat M_j\) coincides with the
exact transition matrix. If some coordinate has entered the transition set earlier, we shall call the
orbit {\em bad}; in that case the scalar factor has already vanished by
Lemma~\ref{lem:bad-stage-kills-H}.

It is important that \(\widehat M_j\) appears only as a bookkeeping device in the identity below.
The actual recursive network does \emph{not} compute the products
\(\chi_{q_1,n}\chi_{q_2,n}\)
as separate bilinear functions. Instead, the branch states are gated successively by nested
applications of the product gadget from \S\ref{subsec:product-gadget}.

\begin{lemma}[Modified cascade identity]
\label{lem:2d-modified-cascade}
For a coordinate index \(\ell\in\{1,\dots,L_1L_2\}\), define
\[
g_{n,\ell}(x,y):=e_\ell^\top G_n(x,y),
\qquad (x,y)\in[0,1]^2.
\]
Then
\begin{equation}
\label{eq:mod-cas}
g_{n,\ell}(x,y)
=
H(R_n(x,y))\,
e_\ell^\top \widehat M_1(x,y)\cdots \widehat M_n(x,y)\,b_{11},
\qquad (x,y)\in[0,1]^2.
\end{equation}
\end{lemma}

\begin{proof}
Corollary~\ref{cor:iterated-cascade-2d} gives
\[
G_n(x,y)=T_{Q_1(x,y)}T_{Q_2(x,y)}\cdots T_{Q_n(x,y)}\,G(R_n(x,y)).
\]
Since \(G=H\,b_{11}\), this becomes
\[
g_{n,\ell}(x,y)
=
H(R_n(x,y))\,
e_\ell^\top T_{Q_1(x,y)}T_{Q_2(x,y)}\cdots T_{Q_n(x,y)}\,b_{11}.
\]

We distinguish two cases.

\noindent\emph{Case 1: Good orbit.}
Assume that \(r^{j-1}(x)\notin J_n\) and \(r^{j-1}(y)\notin J_n\) for every \(j=1,\dots,n\).
Then, by the defining property of the selectors, we have
\begin{align*}
\chi_{0,n}\!\bigl(z^{(1)}_{j-1}(x,y)\bigr)
&=
\chi_{[0,1/2)}\!\bigl(r^{j-1}(x)\bigr),
&
\chi_{1,n}\!\bigl(z^{(1)}_{j-1}(x,y)\bigr)
&=
\chi_{[1/2,1]}\!\bigl(r^{j-1}(x)\bigr),
\\
\chi_{0,n}\!\bigl(z^{(2)}_{j-1}(x,y)\bigr)
&=
\chi_{[0,1/2)}\!\bigl(r^{j-1}(y)\bigr),
&
\chi_{1,n}\!\bigl(z^{(2)}_{j-1}(x,y)\bigr)
&=
\chi_{[1/2,1]}\!\bigl(r^{j-1}(y)\bigr).
\end{align*}
Since \(r^{j-1}(x),r^{j-1}(y)\notin J_n\) on a good orbit, and \(J_n\) contains the overlap point
\(\frac12\), exactly one selector is active in each coordinate. Hence exactly one pair \((q_1,q_2)\)
is active, namely \((q_1,q_2)=Q_j(x,y)\),
and therefore
\[
\widehat M_j(x,y)=T_{Q_j(x,y)},
\qquad j=1,\dots,n.
\]
Substituting this into the exact cascade formula gives the claim.

\noindent\emph{Case 2: Bad orbit.}
Assume that \(r^{j-1}(x)\in J_n\) or \(r^{j-1}(y)\in J_n\) for some \(j\in\{1,\dots,n\}\).
Then Lemma~\ref{lem:bad-stage-kills-H} gives
\(H(R_n(x,y))=0\).
Therefore the right hand side of \eqref{eq:mod-cas} vanishes,
and so does $g_{n,\ell}$ due to \eqref{eq:G-n}.
\end{proof}

\subsection{Recursive realization on the unit square}
\label{subsec:unit-square-recursion}

We now turn the modified cascade identity into an explicit recursive network construction.

Fix a coordinate index \(\ell\in\{1,\dots,L_1L_2\}\).
We initialize the backward state by
\[
\Phi_0(x,y):=H(R_n(x,y))\,e_\ell.
\]
Further, define
\[
B:=\max_{q\in\{0,1\}^2}\|T_q^\top\|_{\infty\to\infty},
\qquad
M_H:=\|H\|_{L^\infty([0,1]^2)},
\qquad
a_n:=\max\{1,B^nM_H\}.
\]
For \(j=1,\dots,n\), define recursively
\[
\widehat\Phi_j(x,y)
:=
\sum_{q_1,q_2\in\{0,1\}}
\Pi_{a_n}\!\left(
\chi_{q_2,n}(z^{(2)}_{j-1}(x,y)),
\,
\Pi_{a_n}\!\left(
\chi_{q_1,n}(z^{(1)}_{j-1}(x,y)),
\,
T_{(q_1,q_2)}^\top \widehat\Phi_{j-1}(x,y)
\right)\right),
\]
with
\(\widehat\Phi_0(x,y):=\Phi_0(x,y)\).

Here \(\Phi_0\) is obtained by composing the scalar realization of \(H(R_n(\cdot))\) from
Theorem~\ref{thm:special-atom-scalar-factor} with the fixed linear embedding
\(s\mapsto s\,e_\ell\), while each map \(y\mapsto T_{(q_1,q_2)}^\top y\) is a fixed affine map on
\(\R^{L_1L_2}\). Thus the only nontrivial part of the recursive construction is the exact gating of
the branch contributions. The point of the nested definition above is that no exact multiplication of
two variable selector outputs is ever required: on a good orbit the selectors are already binary, so
exactly one digit pair survives; on a bad orbit the initial state is already zero, and the identity
\(\Pi_{a_n}(\lambda,0)=0\)
annihilates all subsequent branches automatically.

\begin{theorem}[Recursive realization on the unit square]
\label{thm:2d-special-atom-unit-square}
For every \(\ell\in\{1,\dots,L_1L_2\}\) and every \((x,y)\in[0,1]^2\), we have
\[
e_\ell^\top \Vec(V^nH)(x,y)=b_{11}^\top \widehat\Phi_n(x,y).
\]
Consequently, there exist constants \(C_0,C_1>0\), depending only on the support window, the mask
support, and the fixed local complexity of \(H\), such that
\[
\Vec(V^nH)\in \Ups_{C_0,C_1 n}(\ReLU;2,L_1L_2),
\qquad n\ge 1.
\]
Moreover, the corresponding network weights and biases may be chosen with magnitudes bounded by
\(C_2\Lambda^n\) for suitable constants \(C_2,\Lambda>0\) depending only on the support window,
the mask support, and \(H\).
\end{theorem}

\begin{proof}
We argue by a global good/bad orbit dichotomy.

\noindent\emph{Good orbit.}
Assume that neither coordinate enters the transition set at any stage, namely
\[
r^{j-1}(x)\notin J_n
\quad\text{and}\quad
r^{j-1}(y)\notin J_n
\qquad\text{for all }j=1,\dots,n.
\]
Then the selectors recover the exact dyadic digits in each coordinate, so at every stage exactly one
pair \((q_1,q_2)\) is active, namely
\((q_1,q_2)=Q_j(x,y)\).
Hence the recursive block reduces to the exact adjoint update
\[
\Phi_j(x,y):=T_{Q_j(x,y)}^\top \Phi_{j-1}(x,y),
\qquad j=1,\dots,n,
\]
with \(\Phi_0(x,y)=H(R_n(x,y))e_\ell\).
Now Lemma~\ref{lem:2d-modified-cascade} yields
\[
e_\ell^\top \Vec(V^nH)(x,y)=b_{11}^\top \Phi_n(x,y).
\]
Moreover, we have
\[
\|\Phi_j(x,y)\|_\infty
\le B^jM_H
\le a_n,
\qquad j=0,\dots,n.
\]
Thus every input to the product gadget lies in the admissible range, and the identities from
Lemma~\ref{lem:product-gadget} apply. Consequently, we have
\(\widehat\Phi_j(x,y)=\Phi_j(x,y)\) for \(j=0,\dots,n\), and thus
\[
e_\ell^\top \Vec(V^nH)(x,y)=b_{11}^\top \widehat\Phi_n(x,y).
\]

\noindent\emph{Bad orbit.}
Assume that at least one coordinate enters \(J_n\) at some stage. Then
Lemma~\ref{lem:bad-stage-kills-H} gives
\(H(R_n(x,y))=0\), and so
\[
\widehat\Phi_0(x,y)=\Phi_0(x,y)=0.
\]
As
\(\Pi_{a_n}(\cdot,0)=0\),
every term in the recursive definition of \(\widehat\Phi_j(x,y)\) vanishes whenever
\(\widehat\Phi_{j-1}(x,y)=0\). By induction, we infer
\(\widehat\Phi_j(x,y)=0\) for all \(j=0,\dots,n\).
On the other hand, since
\[
G(R_n(x,y))=H(R_n(x,y))\,b_{11}=0,
\]
the exact cascade also vanishes:
\(\Vec(V^nH)(x,y)=0\).
Thus we conclude that
\[
e_\ell^\top \Vec(V^nH)(x,y)=b_{11}^\top \widehat\Phi_n(x,y)=0
\]
on every bad orbit as well.
This proves the scalar identity for each fixed \(\ell\). Since the number of coordinates is fixed, we
may run the finitely many coordinate constructions in parallel. Thus we conclude
\(\Vec(V^nH)\in \Ups_{C_0,C_1 n}(\ReLU;2,L_1L_2)\).

Finally, the only \(n\)-dependent affine coefficients in the construction come from the selectors and
from the gate scale \(a_n\). The selector slopes are of order
\(\delta_n^{-1}\asymp 2^n\),
while
\(a_n\le \max\{1,B^nM_H\}\).
Hence all weights and biases may be bounded in magnitude by \(C_2\Lambda^n\) for suitable constants
\(C_2,\Lambda>0\).
\end{proof}

\subsection{Clamped gluing and the special-atom theorem}
\label{subsec:clamped-gluing}

We now pass from the vectorized realization on the reference square to the physical function on the
support window. Theorem~\ref{thm:2d-special-atom-unit-square} realizes \(\Vec(V^nH)\) on
\([0,1]^2\), encoding all unit-square patches of \(V^nH\) simultaneously. To recover the scalar
function \(V^nH\) on \(\R^2\), one must glue these compatible patches across the preserved support
window. The next lemma provides this fixed-overhead clamped gluing step, and the resulting theorem
yields the exact realization of \(V^nH\) itself.

\begin{lemma}[Exact clamped gluing on a rectangular window]
\label{lem:exact-clamped-gluing}
Let \(f:\R^2\to\R\) be supported in \([0,L_1]\times[0,L_2]\), and for integers
\(1\le a\le L_1\), \(1\le b\le L_2\), let
\[
f_{a,b}(u,v):=f(u+a-1,v+b-1),
\qquad (u,v)\in[0,1]^2,
\]
denote its unit-square pieces. Assume that each \(f_{a,b}\) belongs to
\(\Ups_{W,L}(\ReLU;2,1)\),
and that the usual edge compatibilities hold:
\[
f_{a,b}(1,v)=f_{a+1,b}(0,v)
\quad (1\le a<L_1),
\]
\[
f_{a,b}(u,1)=f_{a,b+1}(u,0)
\quad (1\le b<L_2).
\]
Then there exist constants \(C_0,C_1>0\), depending only on \(L_1,L_2\), such that
\(f\in \Ups_{C_0W,\;L+C_1}(\ReLU;2,1)\).
\end{lemma}

\begin{proof}
By \cite[Lemma~3.13]{loop}, if \(g_1,\dots,g_m:[0,1]\to\R\) satisfy
\[
g_1(0)=0,\qquad g_m(1)=0,\qquad g_k(1)=g_{k+1}(0)\quad (k=1,\dots,m-1),
\]
then the function
\[
\mathcal G[g_1,\dots,g_m](t)
:=
g_1(\sigma_1(t))
+
\sum_{k=2}^m\bigl(g_k(\sigma_k(t))-g_k(0)\bigr)
\]
agrees with \(g_{k_0}(t-k_0+1)\) on \([k_0-1,k_0]\) and vanishes outside \([0,m]\).
Here \(\sigma_k(t):=\ReLU(t-k+1)-\ReLU(t-k)\) is the standard ramp.

We apply this lemma twice.
First, fix \(b\in\{1,\dots,L_2\}\) and \(v\in[0,1]\). Apply the one-dimensional gluing lemma in the
\(x\)-variable to the compatible family
\(u\mapsto f_{1,b}(u,v),\ \dots,\ u\mapsto f_{L_1,b}(u,v)\).
The horizontal edge compatibilities
\(f_{a,b}(1,v)=f_{a+1,b}(0,v)\)
show that the hypotheses are satisfied, and the support condition on \(f\) gives the required boundary
vanishing at the outer edges. Hence the row-glued function
\[
G_b(x,v)
:=
f_{1,b}(\sigma_1(x),v)
+
\sum_{a=2}^{L_1}\bigl(f_{a,b}(\sigma_a(x),v)-f_{a,b}(0,v)\bigr)
\]
satisfies
\[
G_b(x,v)=f_{a_0,b}(x-a_0+1,v)
\qquad\text{for }x\in[a_0-1,a_0].
\]

Next, fix \(x\in\R\). Apply the same one-dimensional gluing lemma in the \(y\)-variable to the
compatible family
\(v\mapsto G_1(x,v),\ \dots,\ v\mapsto G_{L_2}(x,v)\).
The vertical edge compatibilities for the original patches imply
\(G_b(x,1)=G_{b+1}(x,0)\) for \(b=1,\dots,L_2-1\),
and again the support condition gives the outer boundary vanishing. Therefore
\[
f^\sharp(x,y)
:=
G_1(x,\sigma_1(y))
+
\sum_{b=2}^{L_2}\bigl(G_b(x,\sigma_b(y))-G_b(x,0)\bigr)
\]
satisfies
\[
f^\sharp(x,y)=G_{b_0}(x,y-b_0+1)
\qquad\text{for }y\in[b_0-1,b_0].
\]

Combining the two identities, if
\(x\in[a_0-1,a_0]\) and \(y\in[b_0-1,b_0]\),
then
\[
f^\sharp(x,y)
=
G_{b_0}(x,y-b_0+1)
=
f_{a_0,b_0}(x-a_0+1,y-b_0+1)
=
f(x,y).
\]
Thus \(f^\sharp=f\) on the support window, and the same gluing lemma shows that \(f^\sharp\) vanishes
outside \([0,L_1]\times[0,L_2]\). Hence \(f^\sharp=f\) on all of \(\R^2\).

Finally, the construction uses only fixed ramp maps, composition with the already-constructed patch
networks, subtraction of fixed edge traces, and finite summation over the fixed index sets
\(1\le a\le L_1\), \(1\le b\le L_2\). Since \(L_1,L_2\) are fixed, this enlarges the width only by a
constant factor and adds only \(O(1)\) extra depth. Therefore we have
\(f\in \Ups_{C_0W,\;L+C_1}(\ReLU;2,1)\).
\end{proof}

We now pass from the vectorized realization of \(V^nH\) on the reference square to the physical
function \(V^nH\) on the support window by applying the clamped gluing lemma.

\begin{theorem}[Special-atom theorem]
\label{thm:special-atom-theorem}
Let \(H\) be a special atom. Then there exist constants \(C_0,C_1>0\), depending only on the
support window, the mask support, and the fixed local complexity of \(H\), such that
\(V^nH \in \Ups_{C_0,C_1 n}(\ReLU;2,1)\) for \(n\ge 1\).
\end{theorem}

\begin{proof}
Theorem~\ref{thm:2d-special-atom-unit-square} gives
\(\Vec(V^nH)\in \Ups_{C_0',C_1'n}(\ReLU;2,L_1L_2)\).
The unit-square pieces are compatible by construction, so
Lemma~\ref{lem:exact-clamped-gluing} reconstructs \(V^nH\) from its vectorization with only
\(O(1)\) extra overhead. This gives the stated result.
\end{proof}

\section{General compactly supported seeds}
\label{sec:general-seeds}

In this section we pass from the special-atom theorem to arbitrary compactly supported \(\CPwL\)
seeds. The argument has three finite steps: decompose the seed into finitely many translated
special atoms, move the corresponding realizations to the correct locations by translation
covariance, and sum the resulting networks. Since the number of atoms and their local
complexities depend only on the fixed geometric and combinatorial complexity of the seed, all of
this overhead is independent of the refinement depth \(n\).

\subsection{Finite special-atom decomposition}
\label{subsec:atomic-decomposition}

We begin by decomposing an arbitrary compactly supported \(\CPwL\) seed into finitely many
translated special atoms. The only issue is to choose a sufficiently fine local basis so that each
basis function fits strictly inside a translate of the reference square with margin \(\varrho\).

\begin{proposition}[Finite special-atom decomposition]
\label{prop:adaptive-special-atom-decomp}
Let $g:\R^2\to\R$ be compactly supported and $\CPwL$, and assume
\(\supp g\subset [0,L_1]\times[0,L_2]\).
Then there exist an integer $N\ge 1$, coefficients $a_1,\dots,a_N\in\R$, translation vectors
\(\delta_1,\dots,\delta_N\in\R^2\),
and special atoms $H_1,\dots,H_N$ such that
\begin{equation}
\label{eq:adaptive-atom-decomp}
g(z)=\sum_{\nu=1}^N a_\nu\,H_\nu(z-\delta_\nu),
\qquad z\in\R^2.
\end{equation}
Moreover, the number \(N\) and the local polygonal complexities of the atoms \(H_\nu\) can be
bounded in terms of the fixed support window and a chosen quantitative \(\CPwL\) description of
\(g\).
\end{proposition}

\begin{proof}
Choose a finite triangulation $\mathcal T$ of the support window
\([0,L_1]\times[0,L_2]\)
such that $g$ is affine on every triangle of $\mathcal T$. Since $g$ is compactly supported and
$\CPwL$, such a triangulation exists.

We now refine $\mathcal T$ further, if necessary, so that the support of every nodal basis function of
the refined triangulation has $\ell^\infty$-diameter strictly smaller than $1-2\varrho$. Since the
support window is compact and the triangulation is finite, this can be achieved by finitely many
barycentric subdivisions or any other standard finite local refinement procedure.

Let
\(\phi_1,\dots,\phi_N\)
denote the nodal hat functions of the resulting refined triangulation. Then each $\phi_\nu$ is
nonnegative, compactly supported, and $\CPwL$, and the standard nodal expansion gives
\[
g(z)=\sum_{\nu=1}^N a_\nu\,\phi_\nu(z),
\]
for suitable coefficients $a_\nu\in\R$.

Fix $\nu\in\{1,\dots,N\}$. Since the support of $\phi_\nu$ has $\ell^\infty$-diameter strictly less
than $1-2\varrho$, there exists a translation vector $\delta_\nu\in\R^2$ such that
\(\supp\phi_\nu\subset \delta_\nu+[\varrho,1-\varrho]^2\).
Define
\[
H_\nu(w):=\phi_\nu(w+\delta_\nu),
\qquad w\in\R^2.
\]
Then $H_\nu$ is nonnegative, compactly supported, and $\CPwL$, and
\(\supp H_\nu\subset[\varrho,1-\varrho]^2\).
Hence $H_\nu$ is a special atom in the sense of Definition~\ref{def:special-atom}. Moreover,
substituting \(\phi_\nu(z)=H_\nu(z-\delta_\nu)\) into the nodal expansion of $g$ gives \eqref{eq:adaptive-atom-decomp}.

Finally, the number $N$ is exactly the number of vertices of the refined triangulation, and the local
polygonal complexity of each $H_\nu$ is controlled by the valence structure of that triangulation.
Thus both quantities are bounded in terms of the chosen quantitative $\CPwL$ description of $g$
together with the fixed support window.
\end{proof}

\begin{remark}
The bound on \(N\) in Proposition~\ref{prop:adaptive-special-atom-decomp} depends in general on
geometric scale as well as combinatorics: a function may be simple in terms of affine pieces but still
require many atoms if its support is large relative to the unit-square atom size.
\end{remark}

\subsection{Translation covariance}
\label{subsec:translation-covariance}

The homogeneous operator enjoys the obvious translation covariance. This is the mechanism that
allows us to treat translated special atoms term by term.

\begin{lemma}[Translation covariance]
\label{lem:translation-covariance-2d}
Let $\delta=(\delta_1,\delta_2)\in\R^2$, and define
\(g_\delta(z):=g(z-\delta)\) for \(z\in\R^2\).
Then for every $n\ge 1$, we have
\begin{equation}
\label{eq:translation-covariance}
V^n g_\delta(z)=\bigl(V^n g\bigr)(z-2^{-n}\delta),
\qquad z\in\R^2.
\end{equation}
\end{lemma}

\begin{proof}
Writing $z=(x,y)$ and $\delta=(\delta_1,\delta_2)$, we have
\begin{align*}
(Vg_\delta)(x,y)
&=
\sum_{(j,k)\in\Z^2} c_{j,k}\,
g_\delta(2x-j,\,2y-k)\\
&=
\sum_{(j,k)\in\Z^2} c_{j,k}\,
g(2x-j-\delta_1,\,2y-k-\delta_2)\\
&=
(Vg)\big(x-\frac{\delta_1}{2},\,y-\frac{\delta_2}{2}\big).
\end{align*}
Iterating this relation proves the general identity.
\end{proof}

\subsection{Proof of the main theorem}
\label{subsec:proof-main}

We can now deduce the full theorem.

\begin{proof}[Proof of Theorem~\ref{thm:main}]
Let $g:\R^2\to\R$ be compactly supported and $\CPwL$, with
\(\supp g\subset[0,L_1]\times[0,L_2]\).
By Proposition~\ref{prop:adaptive-special-atom-decomp}, there exist an integer $N\ge 1$,
coefficients $a_1,\dots,a_N\in\R$, translation vectors $\delta_1,\dots,\delta_N\in\R^2$, and special
atoms $H_1,\dots,H_N$ such that
\[
g(z)=\sum_{\nu=1}^N a_\nu\,H_\nu(z-\delta_\nu).
\]
Since $V$ is linear, it follows that
\[
V^n g(z)=\sum_{\nu=1}^N a_\nu\,V^n\bigl[H_\nu(\,\cdot\,-\delta_\nu)\bigr](z).
\]
Fix one term. By Lemma~\ref{lem:translation-covariance-2d}, we have
\[
V^n\bigl[H_\nu(\,\cdot\,-\delta_\nu)\bigr](z)
=
\bigl(V^nH_\nu\bigr)(z-2^{-n}\delta_\nu).
\]
On the other hand, Theorem~\ref{thm:special-atom-theorem} gives constants
\(C_{0,\nu},\,C_{1,\nu}>0\)
depending only on the fixed support window, the mask support, and the local complexity of
$H_\nu$, such that
\[
V^nH_\nu\in \Ups_{C_{0,\nu},C_{1,\nu}n}(\ReLU;2,1),
\qquad n\ge 1.
\]
By the closure property in Remark~\ref{rem:standard-closure}(iii), the translated function
\(z\mapsto \bigl(V^nH_\nu\bigr)(z-2^{-n}\delta_\nu)\)
also belongs to a class
\(\Ups_{\widetilde C_{0,\nu},\widetilde C_{1,\nu}n}(\ReLU;2,1)\)
with constants depending only on the same fixed data.

Since the number of summands $N$ is finite and independent of $n$, we may run all these networks
in parallel and then sum them. By Remark~\ref{rem:standard-closure}(ii), this enlarges the width by
only a constant factor depending on $N$ and the atom complexities, while preserving the linear
depth growth. Therefore there exist constants $C_0,C_1>0$, depending only on the mask support,
the preserved support window, and the quantitative $\CPwL$ complexity of $g$, such that
\(V^n g\in \Ups_{C_0,C_1 n}(\ReLU;2,1)\) for \(n\ge 1\).
This establishes Theorem~\ref{thm:main}.
\end{proof}

\begin{remark}
\label{rem:main-proof-end}
The proof of Theorem~\ref{thm:main} separates naturally into two layers. Section~\ref{sec:recursive-block}
contains the core multivariate construction, combining the exact torus controller, the four-branch
scalar readout, the selector matrices, the recursive block, and the clamped gluing to prove the
special-atom theorem. Section~\ref{sec:general-seeds} then passes to arbitrary compactly supported
\(\CPwL\) seeds by combining that theorem with the finite special-atom decomposition, translation
covariance, and finite summation.
\end{remark}

\section{Conclusions and outlook}
\label{sec:conclusions}

We have proved an exact ReLU realization theorem for two-dimensional tensor-product dyadic scalar
refinement iterates. Under a fixed support-window hypothesis, every compactly supported \(\CPwL\)
seed \(g:\R^2\to\R\) gives rise to iterates \(V^n g\) that admit exact ReLU realizations of fixed width
and depth \(O(n)\). This provides a first genuinely two-dimensional extension of the exact realization
theory for refinement cascades.

The main new ingredient is the exact torus controller for the residual dynamics. In the
one-dimensional loop-controller framework, the residual orbit is transported by an exact forward
state on a polygonal loop. In the present tensor-product setting, the natural controller space is the
product of two such loops. Thus the forward residual transport remains exact, while the remaining
seam ambiguity is confined to the terminal readout and selector stage. The genuinely new
multivariate work then lies in the geometric part of the construction: the digit-pair selectors, the
use of compactly supported interior atoms, and the gluing of squarewise realizations on the support
window.

The tensor-product dyadic case is, in our view, the right first multivariate setting. Its residual
dynamics is still coordinatewise, so the controller architecture is a direct product of one-dimensional
loop controllers. At the same time, the proof identifies a controller mechanism that should extend
more broadly: exact forward transport of the residual state, seam resolution by complementary
readouts, selector-driven propagation of the finite-dimensional cascade, and final patch gluing on
the support window.

Several extensions therefore appear to be natural continuations of the present result. One class
consists of tensor-product dilations in higher dimension, with diagonal expansion matrix
\[
D=\operatorname{diag}(M_1,\dots,M_d),\qquad M_i\ge 2.
\]
There the residual dynamics is again coordinatewise, so one expects the controller space to be a
product of one-dimensional loop controllers, with the new work lying mainly in the combinatorics
of the selector branches and the higher-dimensional gluing. A second class consists of vector-valued
outputs, where the same controller idea should interact with block transition matrices much as in
the one-dimensional homogeneous vector-valued theory. A third class consists of certain non-diagonal
expanding matrices for which
\[
A^p=D
\]
for some \(p\ge 1\) and some diagonal matrix \(D\). In such situations it is natural to expect that
the refinement dynamics can be organized in blocks of length \(p\), reducing the residual transport
after finitely many steps to a diagonal one. This includes quincunx-type examples among the most
natural cases to examine.

We do not pursue these broader settings here. Each would require additional notation, charting,
and bookkeeping, and in some cases a dedicated exposition would be preferable to a compressed
treatment appended to the present paper. Nevertheless, the present theorem should be viewed not
only as a complete result in the scalar two-dimensional dyadic case, but also as a foundational
multivariate instance of the loop-controller method, with higher-dimensional diagonal, vector-valued,
and partially reducible non-diagonal settings as next steps built on the same underlying
idea.

\end{document}